\documentclass{article}

\usepackage{multicol,graphicx,color}
\usepackage{pslatex}
\usepackage{authblk}
\usepackage{amsthm}
\usepackage{amsmath}
\usepackage{amssymb}
\usepackage{latexsym}
\usepackage{lscape}
\usepackage{epsfig}
\usepackage{pstricks}
\usepackage{amsfonts}
\usepackage{appendix}

\newcommand{\myfig}[3][0]{
\begin{center}
  \vspace{0.2cm}
  \includegraphics[width=#3\hsize,angle=#1]{#2}

  \nobreak\medskip
\end{center}
}

\newcommand{\mycaption}[1]{
  \vspace{0.2cm}
  \begin{quote}
    {{\sc Figure} \arabic{figure}: #1}
  \end{quote}
  \vspace{0.2cm}
  \stepcounter{figure}
}

\theoremstyle{plain}
\newtheorem{theorem}{Theorem}

\newtheorem{lemma}[theorem]{Lemma}
\newtheorem{proposition}[theorem]{Proposition}

\numberwithin{equation}{section}
\numberwithin{theorem}{section}

\author[1]{Duokui Yan  \footnote{duokuiyan@buaa.edu.cn. The research of Duokui Yan is supported in part by NSFC (No. 11101221).}} \affil[1]{School of Mathematics and System Science, Beihang University, Beijing 100191, P. R. China}

\author[1]{Rongchang Liu \footnote{22201209001021000000@smss.buaa.edu.cn}}

\author[2]{Geng-zhe Chang \footnote{changgz@ustc.edu.cn}} \affil[2]{School of Mathematical Science, University of science and Technology of China, Hefei, Anhui 230026, P. R. China}

\title{A type of multiple integral with loggamma function}

\date{}

\begin{document}

\maketitle
\begin{abstract}
 In this paper, we study the multiple integral $ \displaystyle I= \int_0^1 \int_0^1 \dots \int_0^1 f(x_1+x_2 + \dots +x_n) \, dx_1 \, dx_2 \, \dots \, dx_n$. A general formula of $I$ is presented. As an application, the integral $I$ with $f(x)= \log \Gamma(x)$ is evaluated. We show that the values of $I$ share a common formula for all $n \in \mathbb{N}$. The subsidiary computational challenges are substantial and interesting in their own right. 
\end{abstract}

%

\section{Introduction}

A general idea, when faced with a multiple integral, is to lower its dimension. For example \cite{CH2}, a well-known n-dimensional integral
\begin{equation}\label{intoftriangle}
 \underset{ \substack{ x_1+x_2+ \dots +x_n \leq 1\\ x_1, x_2, \dots, x_n \geq 0 }  }{\idotsint}     f(x_1+ x_2 + \dots + x_n) dx_1 dx_2 \dots d x_n ,
 \end{equation}
can be simplified to an one-dimensional integral
\[ \frac{1}{(n-1)!} \int_0^1 t^{n-1} f(t) dt. \]

However, to the best of our knowledge, a similar integral
\begin{equation}\label{goal}
I=  \int_0^1 \int_0^1 \dots \int_0^1 f(x_1+x_2 + \dots +x_n) dx_1 \, dx_2 \, \dots \, dx_n
 \end{equation}
has no such formula. 

The goal of this paper is to find a formula for the above integral $I$ and apply it to a special case when $f(x)= \log \Gamma (x)$. The main results are as follows. A general formula of $I$ is obtained in Theorem \ref{main01}. A detailed proof can be found in Theorem \ref{maintheorem1} of Section \ref{general}.
\begin{theorem}\label{main01}
The integral $I$ satisfies the following formula: 
\begin{equation}
\begin{split}
I&= \int_0^1 \int_0^1 \dots \int_0^1 f(x_1+x_2 + \dots +x_n) dx_1 \, dx_2 \, \dots \, dx_n \\
 & =  \frac{1}{(n-1)!} \sum_{m=1}^{n} \int_0^1 G_m(t) f(t+m-1)  dt, 
\end{split}
\end{equation}
where 
\[G_m(t)= \sum_{i=1}^{m} (-1)^{i-1}(t+m-i)^{n-1}  {n \choose i-1}. \]
\end{theorem}
If $f(x)= \log \Gamma(x)$, the evaluation of $I$ has a simple formula as in Theorem \ref{main02}. The main challenge is to choose appropriate combinatorial identities to simplify the value of $I$. A proof of it can be found in Theorem \ref{maintheorem2} of Section \ref{application}. 
\begin{theorem}\label{main02}
\begin{eqnarray}\label{loggammaformula}
I&=&I(n)= \int_0^1 \int_0^1 \dots \int_0^1 \log \Gamma(x_1+x_2 + \dots +x_n) dx_1 \, dx_2 \, \dots \, dx_n \nonumber \\
&=&\frac{1}{2} \log (2 \pi) - \frac{n-1}{2}H_n + \sum_{k=2}^{n-1} \frac{  (-1)^{n+k+1} k^n }{n!} {n-1 \choose k} \log k, \nonumber
\end{eqnarray}
where $\displaystyle H_n=\sum_{k=1}^{n} \frac{1}{k} =1+\frac{1}{2}+ \dots + \frac{1}{n}$.
\end{theorem}

The paper is organized as follows. Sections \ref{simple2} and \ref{simple3} study the lower dimension cases. Ideas are explained clearly for $n=2$ and $3$. One can see from Figure 1 and Figure 2 about how we cut the square or the cube so that the integral $I$ over each subset becomes a simple one-dimensional integral. In Section \ref{general}, a formula of the integral $I$ is derived in Theorem \ref{maintheorem1}. As its application, Section \ref{application} evaluates the integral $I$ with $f(x)= \log \Gamma(x)$. 

\section{Simple case: n=2}\label{simple2}
We start from the simple case: $\displaystyle \int_0^1  \int_0^1  f(x+y) dx \, dy $. The integral domain is a unit square, and the integrand is $f(x+y)$. Let $t=x+y$.  The unit square can then be divided into two domains: $D_1$ and $D_2$ as in Figure 1, where
\[D_1= \{(x,y): \ 0 \leq x+y \leq 1, \ \  0 \leq x\leq 1, \ \  0 \leq y\leq 1 \}, \]
\[  D_2= \{(x,y):  \ 1 \leq x+y \leq 2,  \ \ 0 \leq x\leq 1, \ \  0 \leq y\leq 1 \}. \]
\begin{center}

\psset{xunit=1in,yunit=1in}
\begin{pspicture}(-1,-1)(2.6,1.4)
 \psline{->}(-0.5,0)(1.8,0)
 \psline{->}(0.5 ,-0.6)(0.5, 1.2)

   \rput(0.4, 0.5){$1$}

\psdots[dotsize=1.5pt 1](1, 0)
\psdots[dotsize=1.5pt 1](0.5, 0.5)
\psdots[dotsize=1.5pt 1](1, 0.5)

\psline[linewidth=1pt,linecolor=red,linestyle=dotted](0.5, 1)(1.5, 0)
\psline(0.5, 0.5)(1, 0.5)
\psline(1,0.5)(1, 0)
\psline(0.4, 0.6)(1.1, -0.1)
\rput(1, -0.1){$1$}
\rput(1.85  ,-0.05){$x$}
\rput(0.55, 1.25){$y$}
\rput(0.665,  0.165) {$D_1$}

\rput(1.165,  0.165) {$D_3$}
\rput(0.88, 0.38 ){$D_2$}
\rput(0.665, 0.665 ){$D_4$}

\rput(1.2,  0.55){$(1,1)$}

\rput(1.5, -0.1 ){$2$}

\rput(0.4, 1 ){$2$}
 \rput(0.5,-0.8){\textsc{Figure} 1: Domains $D_1$, $D_2$, $D_3$ and $D_4$.}

\end{pspicture}
\end{center}
The following lemma shows that $\int_0^1  \int_0^1  f(x+y) dx \, dy $ is the sum of two one-dimensional integrals.
\begin{lemma}\label{2dcase}
\begin{equation}\label{lemma1}
\begin{split}
\int_0^1  \int_0^1  f(x+y) dx \, dy  =& \iint_{D_1 } f(x+y)  dx  dy + \iint_{D_2 } f(x+y)  dx  dy \\
=& \int_0^1 t f(t) dt   +\int_0^1 (1-t) f(t+1) dt.
\end{split}
\end{equation}
\end{lemma}

\begin{proof}
For integral over domain $D_1$, note that $t=x+y$. Introduce the transformation $(x,y) \mapsto (x, t)$. It is clear that the Jacobian is 1. Then
\begin{equation}\label{2dxy}
\iint_{D_1 } f(x+y)  dx  dy= \int_0^1 \int_0^t  f(t) dx  dt  = \int_0^1 t f(t) dt.
\end{equation}
For integral over domain $D_2$, let $\, x_1=1-x \, $ and $\, y_1=1-y$. Then the new variable $(x_1, y_1)$ is in $D_1$ and
 \begin{equation}\label{2dx1y1}
 \begin{split}
& \iint_{D_2 } f(x+y)  dx  dy= \iint_{D_1 } f(2-x_1-y_1)  dx_1  dy_1 \\
= & \int_0^1 t f(2-t) dt.
\end{split}
\end{equation}
If one set $u=1-t$, it follows that $\displaystyle \int_0^1 t f(2-t) dt= \int_0^1 (1-u) f(u+1) du$. Then
\begin{equation}\label{D2int}
\iint_{D_2 } f(x+y)  dx  dy= \int_0^1 (1-u) f(u+1) du.
\end{equation}
Identities \eqref{2dxy} and \eqref{D2int} then imply \eqref{lemma1}. The proof is complete.

\end{proof}


\section{The case: $n=3$}\label{simple3}
For $n=3$, the main idea is to cut the unit box into several simplexes so that we can apply the integral formula over each simplex. In this section, we explain the case $n=3$ clearly and visualize the division of the unit cube in Figure 2. Let $E= \{(x, y,z): \ 0 \leq x\leq 1, \ \  0\leq y \leq 1, \ \  0 \leq z\leq 1\}$ be the unit cube. Set
\[E_1= \{ (x, y,z): \  0\leq x+y+z\leq 1, \ \  0\leq x\leq 1, \ \ 0\leq y \leq 1, \ \  0\leq z\leq 1 \}, \]
\[E_2= \{ (x, y,z): \  1\leq x+y+z\leq 2, \ \  0\leq x\leq 1, \ \ 0\leq y \leq 1, \ \  0\leq z\leq 1 \}, \]
and
\[E_3= \{ (x, y,z): \ 2 \leq x+y+z\leq 3, \ \  0\leq x\leq 1, \ \ 0\leq y \leq 1, \ \  0\leq z\leq 1 \}. \]
Then $E= E_1 \cup E_2 \cup E_3$ and the integral can be split into three parts:
\[\int_0^1  \int_0^1 \int_0^1  f(x+y+z) dx \, dy \, dz  \]
\[=  \int_{E_1}  f(x+y+z) dx \, dy \, dz +  \int_{E_2}  f(x+y+z) dx \, dy \, dz + \int_{E_3}  f(x+y+z) dx \, dy \, dz . \]
By formula for integral \eqref{intoftriangle}, it follows that $ \displaystyle \int_{E_1}  f(x+y+z) dx \, dy \, dz  =  \frac{1}{2} \int_0^1  t^2 f(t) dt $. The difficult parts are the integrals over regions $E_2$ and $E_3$. The following lemma explains how to simplify these two integrals to one-dimensional integrals.
\begin{lemma}\label{xyz}
\begin{equation}\label{lemma2}
\begin{split}
& \int_0^1  \int_0^1 \int_0^1  f(x+y+z) dx \, dy \, dz  \\
=& \frac{1}{2}  \int_0^1 t^2 f(t) dt +  \frac{1}{2}  \int_0^1 (-2 t^2 +2 t +1)f(t+1) dt +  \frac{1}{2}  \int_0^1 (1-t)^2 f(t+2) dt.
\end{split}
\end{equation}
\end{lemma}
\begin{proof}
We introduce the transformation $(x, y,z) \mapsto (x, y, t)$. By formula for integral \eqref{intoftriangle},
\begin{eqnarray}\label{3dxyz1}
  \int_{E_1}  f(x+y+z) dx \, dy \, dz  =  \frac{1}{2} \int_0^1  t^2 f(t) dt.
 \end{eqnarray}
Note that  integral \eqref{3dxyz1} can be applied to calculate the integral over domain $E_3$. Let $x_1=1-x$, $y_1=1-y$ and $z_1=1-z$. The integral over domain $E_3$
 becomes
 \begin{eqnarray}\label{3dxyz2}
 & & \int_{E_3}  f(x+y+z) dx \, dy \, dz  \nonumber \\
 &=& \int_{E_1}  f(3-x_1-y_1-z_1) dx_1 \, dy_1 \, dz_1  \\
 &=&  \frac{1}{2}  \int_0^1 t^2 f(3-t) dt . \nonumber
 \end{eqnarray}
If one sets $u=1-t$, it implies that $\displaystyle \frac{1}{2} \int_0^1  t^2 f(3-t) dt = \frac{1}{2}  \int_0^1  (1-u)^2 f(2+u) du $.  Hence,
\begin{equation}\label{3DE3int}
 \int_{E_3}  f(x+y+z) dx \, dy \, dz = \frac{1}{2}  \int_0^1  (1-t)^2 f(t+2) dt.
\end{equation}
By equations \eqref{3dxyz1} and \eqref{3DE3int}, Lemma \ref{xyz} is reduced to show
\begin{equation}\label{E2}
  \int_{E_2}  f(x+y+z) dx \, dy \, dz = \frac{1}{2}  \int_0^1 (-2 t^2 +2 t +1)f(t+1) dt .
  \end{equation}
Consider the domain \[E_{20}= \{ (x, y,z): \  1 \leq x+y+z\leq 2, \ \ 0 \leq x\leq 2, \ \ 0 \leq y \leq 2, \ \  0 \leq z\leq 2 \}.\]  Similar to Figure 1, we can divide $E_{20}$ into several domains related to $E_1$ and $E_2$. For each variable $x$, $y$, or $z$, it belongs to either $[0, 1]$ or $[1, 2]$. For example, if $x \in [1, 2]$, since $1 \leq x+y+z \leq 2$, it implies that  $y, z \in [0, 1]$. Then the domain $E_{20}$ can be divided into 4 different subdomains: $E_{21}$, $E_{22}$, $E_{23}$ and $E_2$. A picture of this partition can be seen in Figure 2.
\[E_2= \{ (x, y,z): \  1\leq x+y+z\leq 2, \ \  0\leq x\leq 1, \ \ 0\leq y \leq 1, \ \  0 \leq z\leq 1 \},\]
\[ E_{21}=  \{ (x, y,z): \  1\leq x+y+z\leq 2, \ \ 1\leq x\leq 2, \ \ 0\leq y \leq 1, \ \  0\leq z\leq 1 \},\]
\[ E_{22}=  \{ (x, y,z): \  1 \leq x+y+z\leq 2, \ \ 0 \leq x\leq 1, \ \ 1\leq y \leq 2, \ \  0 \leq z\leq 1 \},\]
and
\[ E_{23}=  \{ (x, y,z): \  1\leq x+y+z\leq 2, \ \ 0 \leq x\leq 1, \ \  0 \leq y \leq 1, \ \  1\leq  z\leq 2 \}.\]
\begin{center}
\myfig{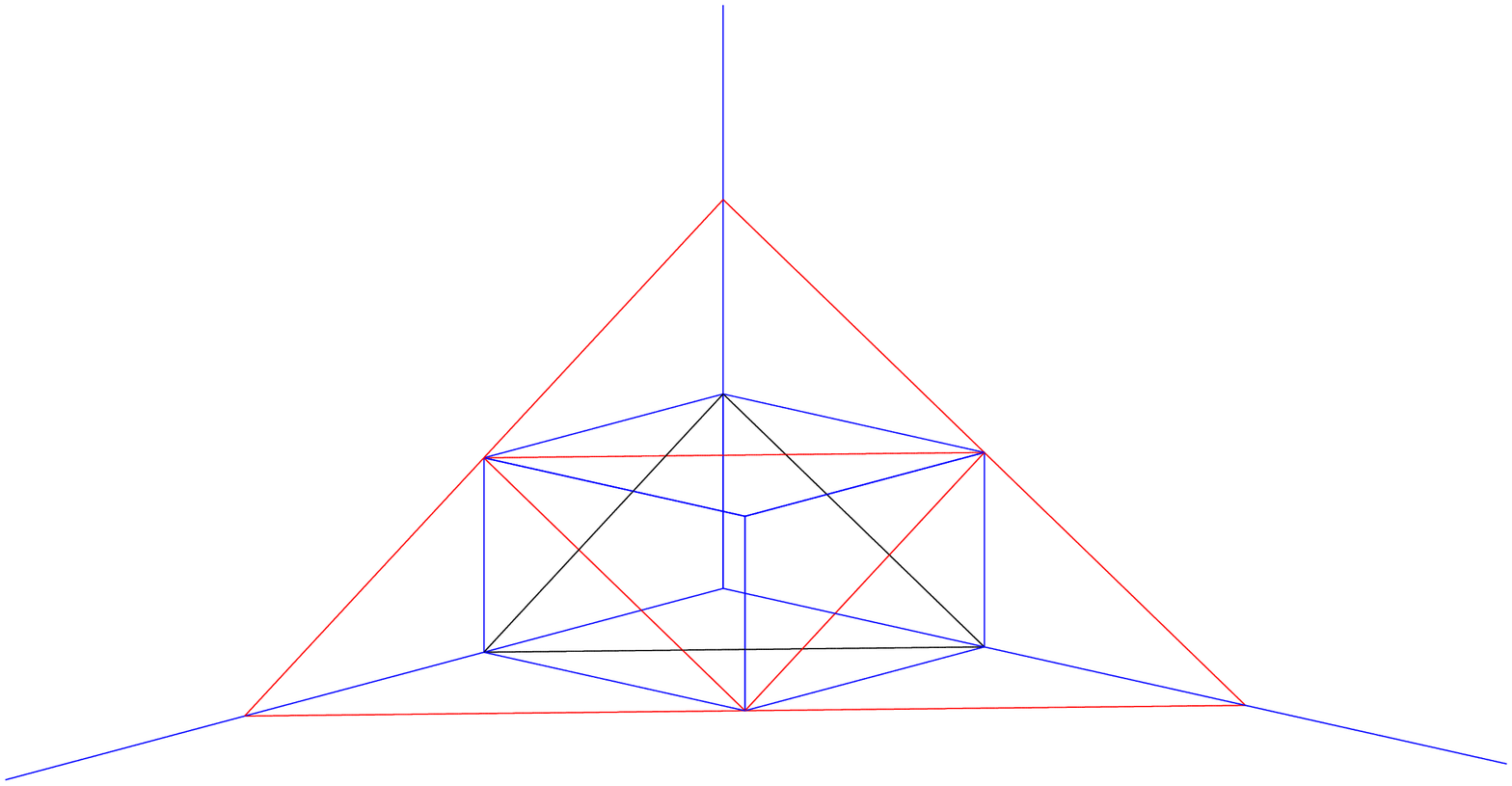}{0.8}
\mycaption{Region $E_{20}$ and its partition: $E_2$, $E_{21}$, $E_{22}$, $E_{23}$.}
\end{center}
Again by formula for integral \eqref{intoftriangle}, the integral over domain $E_{20}$ is
\begin{eqnarray}\label{E20int}
& &\int_{E_{20}}  f(x+y+z) dx \, dy \, dz  \nonumber\\
&=& \int_1^2 \frac{1}{2} t^2 f(t) dt = \frac{1}{2}  \int_0^1  (t+1)^2 f(t+1) dt.
\end{eqnarray}
And it also satisfies
 \begin{eqnarray}\label{E20}
 & & \int_{E_{20}}  f(x+y+z) dx \, dy \, dz  \nonumber \\
 &=&  \int_{E_{21}}  f(x+y+z) dx \, dy \, dz  +   \int_{E_{22}}  f(x+y+z) dx \, dy \, dz  \\
 &  & + \int_{E_{23}}  f(x+y+z) dx \, dy \, dz  +  \int_{E_{2}}  f(x+y+z) dx \, dy \, dz. \nonumber
 \end{eqnarray}
By definitions of $E_{21}$, $E_{22}$ and $E_{23}$, it is clear that $ \displaystyle \int_{E_{21}}  f(x+y+z) dx \, dy \, dz= \int_{E_{22}}  f(x+y+z) dx \, dy \, dz =\int_{E_{23}}  f(x+y+z) dx \, dy \, dz$. So we only need to consider   $ \displaystyle \int_{E_{21}}  f(x+y+z) dx \, dy \, dz$.
Let $\tilde{x} = x-1$, then by integral\eqref{3dxyz1},
\begin{eqnarray}\label{E21}
 & & \int_{E_{21}}  f(x+y+z) dx \, dy \, dz  \nonumber \\
 &=&  \int_{E_{1}}  f(\tilde{x}+y+z+1) d \tilde{x} \, dy \, dz  \\
 &=&    \frac{1}{2}  \int_0^1  t^2 f(t+1) dt. \nonumber
 \end{eqnarray}
Therefore, \eqref{E20int}, \eqref{E20} and \eqref{E21} imply that
\begin{eqnarray*}
& & \int_{E_{2}}  f(x+y+z) dx \, dy \, dz \nonumber \\
&=& \int_{E_{20}}  f(x+y+z) dx \, dy \, dz - 3 \int_{E_{21}}  f(x+y+z) dx \, dy \, dz \nonumber \\
&=& \frac{1}{2}  \int_0^1  (t+1)^2 f(t+1) dt-  \frac{3}{2}  \int_0^1  t^2 f(t+1) dt \\
&=& \frac{1}{2}  \int_0^1 (-2 t^2 +2 t +1)f(t+1) dt.
\end{eqnarray*}
The proof is complete.
\end{proof}

\section{General case}\label{general}
In this section, we study the general n-dimensional case
\[ I=\int_0^1 \int_0^1 \dots \int_0^1 f(x_1+x_2 + \dots +x_n) dx_1 \, dx_2 \, \dots \, dx_n.\]
A general formula of $I$ is calculated in Theorem \ref{maintheorem1}. The rest of this section is to proof it. A recursive formula is derived in  Theorem \ref{recursive} and the proof of Theorem \ref{maintheorem1} follows by Theorem \ref{formulaofim}.
\begin{theorem}\label{maintheorem1}
The integral $I$ satisfies the following formula: 
\begin{equation}
\begin{split}
I&= \int_0^1 \int_0^1 \dots \int_0^1 f(x_1+x_2 + \dots +x_n) dx_1 \, dx_2 \, \dots \, dx_n \\
 & =  \frac{1}{(n-1)!} \sum_{m=1}^{n} \int_0^1 G_m(t) f(t+m-1)  dt, 
\end{split}
\end{equation}
where 
\[G_m(t)= \sum_{i=1}^{m} (-1)^{i-1}(t+m-i)^{n-1}  {n \choose i-1}. \]
\end{theorem}

The idea is to divide the n-dimensional unit box into n different polyhedrons and the integral $I$ over each polyhedron can be simplified to one-dimensional integrals by applying the tricks in the 2D or 3D cases. The  $n$ different polyhedrons are defined as follows
\begin{equation*}
\begin{split}
K_1  & = \{(x_1, x_2,\dots, x_n):  \\
 & \  0 \leq x_1+x_2+ \dots + x_n \leq  1, \ \  0 \leq x_1 \leq 1,\ \  0 \leq x_2 \leq 1,  \ \ \dots, \ \   0 \leq x_n \leq  1 \},
 \end{split}
\end{equation*}
\begin{equation*}
\begin{split}
K_2  & = \{(x_1, x_2,\dots, x_n):  \\
 & \  1 \leq x_1+x_2+ \dots + x_n \leq  2, \ \  0 \leq x_1 \leq 1 ,\ \  0 \leq x_2 \leq 1,  \ \ \dots, \ \   0 \leq x_n \leq 1 \},
 \end{split}
 \end{equation*}
      \[     \dots \dots  \]
 \begin{equation*}
\begin{split}
K_n & = \{(x_1, x_2,\dots, x_n):  \\
 & \  n-1 \leq x_1+x_2+ \dots + x_n \leq  n, \ \  0 \leq x_1 \leq 1,\ \  0 \leq x_2 \leq 1,  \ \ \dots, \ \   0 \leq x_n \leq 1 \}.
 \end{split}
 \end{equation*}
By formula for integral \eqref{intoftriangle}, the integral over the polyhedron $K_1$ satisfies the following proposition.
\begin{proposition}\label{K1}
\[\int_{K_1}  f(x_1+x_2+ \dots + x_n) dx_1 \, dx_2 \, \dots \, dx_n=\frac{1}{(n-1)! }\int_0^1 t^{n-1} f(t) dt. \]
\end{proposition}
Let \[\displaystyle I_m=  \int_{K_m}  f(x_1+x_2+ \dots + x_n) dx_1 \, dx_2 \, \dots \, dx_n, \, \, \,  \quad  \, \, m=1, 2, \dots, n.\]
It is clear that $\displaystyle I = \sum_{m=1}^{n} I_m$. To evaluate the integral $I$, it is equivalent to calculate each $I_m \, (1\leq m\leq n)$. Define
\begin{equation}\label{jsm}
J_{s, m}= \int_{K_s} f(x_1+ \dots + x_n +m-s)  dx_1 \, dx_2 \, \dots \, dx_n,
\end{equation}
 where $s$ is an integer and $1 \leq s \leq m$. It is clear that $J_{m,m}=I_m$. For any $1 \leq s \leq m-1$, $J_{s, m}$ can be calculated from $I_s$. The following theorem shows that $I_m$ satisfies a recursive formula.

\begin{theorem}\label{recursive}
\begin{equation}\label{recursiveim}
I_m =\frac{1}{(n-1)! }\int_{0}^{1}( t+ m-1)^{n-1} f(t+ m-1) dt  - a_1 J_{1, m} -a_2 J_{2,m} - \dots - a_{m-1} J_{m-1, m},
\end{equation}
where
\[a_i= {m+n-i-1 \choose n-1} , \quad i=1,2, \dots, m-1. \]
\end{theorem}
\begin{proof}
We consider the following region:
\begin{equation*}
\begin{split}
K_{m0}  & = \{(x_1, x_2,\dots, x_n):  \\
 & \  m-1 \leq x_1+x_2+ \dots + x_n \leq  m, \ \  0 \leq x_1 \leq m,\ \  0 \leq x_2 \leq m,  \ \ \dots, \ \   0 \leq x_n \leq m\}.
 \end{split}
 \end{equation*}
By Proposition \ref{K1},
 \begin{equation}\label{Km0int001}
 \begin{split}
 \int_{K_{m0}}  f(x_1+x_2+ \dots + x_n) dx_1 \, dx_2 \, \dots \, dx_n   =\frac{1}{(n-1)! }\int_{m-1}^{m} t^{n-1} f(t) dt \\
  = \frac{1}{(n-1)! }\int_{0}^{1}( t+ m-1)^{n-1} f(t+ m-1) dt.
 \end{split}
 \end{equation}
We then define subset $K_{i_1 i_2 \dots i_n}  \subset K_{m0}$ as follows
\begin{equation*}
\begin{split}
K_{i_1 i_2 \dots i_n}  & = \{(x_1, x_2,\dots, x_n):  \\
 & \  m-1 \leq x_1+x_2+ \dots + x_n \leq  m, \\
 & \   i_1-1 \leq x_1 \leq i_1,\ \  i_2-1 \leq x_2 \leq i_2,  \ \ \dots, \ \   i_n-1 \leq x_n \leq i_n \},
 \end{split}
 \end{equation*}
 where $i_1, i_2, \dots, i_n \in [1, m]$ are positive integers. It is clear that the intersection of any two subsets only happens on their boundaries and we need to classify these subsets so that the integral over such subsets can be evaluated easily. Note that by definition,
$K_{1,1, \dots, 1}= K_m$. In order to find the integral over $K_m$, we need to subtract the integrals over all the other nonempty subsets $K_{i_1 i_2 \dots i_n} \,  (i_1, i_2, \dots, i_n \in [1, m])$  from $ \displaystyle \int_{K_{m0}}  f(x_1+x_2+ \dots + x_n) dx_1 \, dx_2 \, \dots \, dx_n $.

First, we classify these nonempty subsets $K_{i_1 i_2 \dots i_n}  \,  (i_1, i_2, \dots, i_n \in [1, m])$.  For any  $K_{i_1 i_2 \dots i_n}$,   we define a transformation
\begin{equation}\label{transxx}
\tilde{x}_1=x_1-(i_1-1), \ \  \tilde{x}_2=x_2-(i_2-1),   \, \ \dots, \ \ \tilde{x}_n=x_n-(i_n-1).
\end{equation}
Then the set $K_{i_1 i_2 \dots i_n} $ becomes
\begin{equation*}
\begin{split}
\widetilde{K}_{i_1 i_2 \dots i_n}  &= \{(\tilde{x}_1 , \tilde{x}_2,\dots, \tilde{x}_n):  \\
 & \  m+n-(i_1+ i_2 + \dots + i_n)-1 \leq \tilde{x}_1+\tilde{x}_2+ \dots + \tilde{x}_n \leq  m+n-(i_1+ i_2 + \dots + i_n), \\
 & \   0 \leq \tilde{x}_1 \leq 1,\ \  0 \leq \tilde{x}_2 \leq 1,  \ \ \dots, \ \  0 \leq \tilde{x}_n \leq 1\}.
 \end{split}
 \end{equation*}
Let $s= m+n-(i_1+ i_2 + \dots + i_n)$. It is clear that $ K_{i_1 i_2 \dots i_n} \cong  \widetilde{K}_{i_1 i_2 \dots i_n}  = K_s$. Since $m+n-s= \sum_{j=1}^{n} i_j \geq n$, it follows that $s \leq m$. Note that if $s=m$, by definition in \eqref{jsm}, $ J_{m, m}= I_m$. If $s=0$, $K_{i_1 i_2 \dots i_n} \cong  \widetilde{K}_{i_1 i_2 \dots i_n}  = \{0\}$.  If $s< 0$, $K_{i_1 i_2 \dots i_n} \cong  \widetilde{K}_{i_1 i_2 \dots i_n}  = \varnothing$. So we only need to consider the case: $1 \leq s\leq m-1$. For any given $s \in [1, m-1]$, it follows that
\begin{eqnarray}\label{Jsmint001}
& & \int_{K_{i_1 i_2 \dots i_n}}  f(x_1+x_2+ \dots + x_n) dx_1 \, dx_2 \, \dots \, dx_n \nonumber \\
&=& \int_{\overline{K}_{i_1 i_2 \dots i_n}}   f(\tilde{x}_1+ \dots + \tilde{x}_n+ i_1+ \dots + i_n -n) d \tilde{x}_1 \dots d \tilde{x}_n \nonumber \\
&=& \int_{K_s} f(x_1+ \dots + x_n +m-s)  dx_1 \dots dx_n  \\
&=& J_{s, m}. \nonumber
\end{eqnarray}
It implies that the subsets $ K_{i_1 i_2 \dots i_n} \, (i_1, i_2, \dots, i_n \in [1, m], i_1+i_2+ \dots +i_n \neq n )$ with nonzero measure can be classified into $m-1$ classes. In each class, every element is identical to some subset $K_s$ after a shifting transformation in \eqref{transxx}: $ (x_1, x_2, \dots, x_n) \mapsto (\tilde{x}_1,\tilde{x}_2, \dots, \tilde{x}_n) $.

Next step is to fix $m$ and $s \ ( m-1 \leq s \leq 1)$, and find out how many subsets are identical to $K_s$. Since $s=  m+n-(i_1+ i_2 + \dots + i_n)$, it follows that
\begin{equation}\label{i1i2insolution}
m+n-s= i_1 + i_2+ \dots + i_n,  \quad \text{where} \, \  i_1, i_2, \dots, i_n  \ \text{are positive integers}.
\end{equation}
The number of positive integer solutions $(i_1, i_2, \dots, i_n)$ for \eqref{i1i2insolution} is ${m+n-s-1 \choose n-1}$. It follows that the total number of subsets identical to $K_s \, ( s\in [1, m-1])$ is
\begin{eqnarray}\label{as}
a_s &=& {m+n-s-1 \choose n-1} .
\end{eqnarray}
Therefore, by \eqref{Km0int001}, \eqref{Jsmint001} and \eqref{as}, integral $I_m$ satisfies
\begin{eqnarray}\label{Kmint}
 I_m&=&  \int_{K_m}  f(x_1+x_2+ \dots + x_n) dx_1 \, dx_2 \, \dots \, dx_n  \nonumber \\
&=&\frac{1}{(n-1)! }\int_{0}^{1}( t+ m-1)^{n-1} f(t+ m-1) dt  \\
 & &      \qquad  \quad    -  a_1 J_{1, m} -a_2 J_{2,m} - \dots - a_{m-1} J_{m-1, m},\nonumber
\end{eqnarray}
where $a_s \ (s=1, \dots, m-1)$  is defined by \eqref{as}. The proof is complete.
\end{proof}
By the cases $n=2, \, 3$ and $4$, we can show by induction that
\begin{equation}\label{formulaim}
 I_m=  \frac{1}{(n-1)!}\int_0^1 G_m(t) f(t+m-1)  dt,
 \end{equation}
where $G_m(t)$ is a polynomial. It follows that
\begin{equation}\label{jsm}
\begin{split}
 J_{s, m}&= \int_{K_s} f(x_1+ \dots + x_n +m-s)  dx_1 \dots dx_n \\
 & =  \frac{1}{(n-1)!} \int_0^1 G_s(t) f(t+m-1)  dt,
 \end{split}
 \end{equation}
 where $s$ is an integer and $1 \leq s \leq m$. The desired integral $I$ satisfies
\begin{equation}\label{formulaofi}
 I= \int_0^1 \int_0^1 \dots \int_0^1 f(x_1+x_2 + \dots +x_n) dx_1 \, dx_2 \, \dots \, dx_n= \sum_{m=1} ^{n} I_m.
 \end{equation}
 In order to find a formula for $I$, the only challenge left is to compute the polynomial $G_{m}(t)$ in equation \eqref{formulaim} for all $1 \leq m \leq n$.  For $m=1, 2, 3$ and $4$, a direct calculation shows that
\begin{equation}\label{gst}
\begin{split}
& G_1(t)= t^{n-1}, \\
& G_2(t)= (t+1)^{n-1}-{n \choose 1}\, t^{n-1}, \\
& G_3(t)= (t+2)^{n-1}- {n \choose 1}\, (t+1)^{n-1}+ {n \choose 2}\,  t^{n-1},\\
& G_4(t)= (t+3)^{n-1}- {n \choose 1} \, (t+2)^{n-1}+ {n \choose 2}\,  (t+1)^{n-1} -  {n \choose 3}\,  t^{n-1}.
\end{split}
\end{equation}
It is reasonable to believe that $G_m(t)$ follows a pattern. The following theorem actually proves this fact.
\begin{theorem}\label{formulaofim}
\begin{equation}\label{gmtformula}
G_m(t) = \sum_{i=1}^{m} (-1)^{i-1}(t+m-i)^{n-1}  {n \choose i-1}.
\end{equation}
\end{theorem}
\begin{proof}
The proof is based on the recursive formula \eqref{recursiveim} in Theorem \ref{recursive} and the identity \eqref{jsm}. By formula \eqref{recursiveim},
\begin{eqnarray*}
I_m&=& \frac{1}{(n-1)! }\int_{0}^{1}( t+ m-1)^{n-1} f(t+ m-1) dt  -\sum_{i=1}^{m-1} a_i J_{i, m}  \\
&=&  \frac{1}{(n-1)! }\int_{0}^{1} G_m(t) f(t+ m-1) dt ,
\end{eqnarray*}
where
\begin{equation}\label{1forgmt}
 G_m(t)= ( t+ m-1)^{n-1}- \sum_{i=1}^{m-1} a_i G_i(t), \quad\text{and} \quad  a_i=  {m+n-i-1 \choose n-1}.
\end{equation}
We show this theorem by induction. It is clear that formula \eqref{gmtformula} of $G_m(t)$ holds for $m=1$. Assume that it holds for any $1 \leq m \leq k$. We need to show that formula \eqref{gmtformula} also holds for $m=k+1$. By Theorem \ref{recursive}, the polynomial $G_{k+1}(t)$ satisfies
\begin{equation}\label{gk1tform}
\begin{split}
& G_{k+1}(t) \\
=&  (t+k)^{n-1} + \sum_{i=1}^{k} {k+1+n-i-1 \choose n-1} \sum_{j=1}^{i} (-1)^{j} (t+i-j)^{n-1}  {n \choose j-1}.
\end{split}
\end{equation}
By formula \eqref{gmtformula}, we can consider each $G_m (t) \, (1\leq m \leq k)$ as a polynomial of $(t+m-j)^{n-1} \, (j=1, 2, \dots, m)$ with coefficient $(-1)^{j-1} {n \choose j-1}$.  Then identity \eqref{gk1tform} implies that the coefficient of $(t+p)^{n-1}$ in $G_{k+1}(t)$ is
\begin{equation}\label{tpngk1}
L_{p}\left(G_{k+1}(t)\right) =\sum_{i=p+1}^{k}  {k+1+n-i-1 \choose n-1}  (-1)^{i-p} {n \choose i-p-1},
\end{equation}
where integer $p \in [0, k-1]$. Similarly, $G_{k}(t)$ satisfies
\begin{eqnarray*}
G_{k}(t) &=&  (t+k-1)^{n-1} + \sum_{i=1}^{k-1} {k+n-i-1 \choose n-1} \sum_{j=1}^{i} (-1)^{j} (t+i-j)^{n-1}  {n \choose j-1}.
\end{eqnarray*}
And the coefficient of  $(t+p)^{n-1} \, (p \in [0, k-2])$ in $G_{k}(t)$ is
\begin{equation}\label{tpngk}
\sum_{i=p+1}^{k-1}  {k+n-i-1 \choose n-1}  (-1)^{i-p} {n \choose i-p-1}.
\end{equation}
Note that $\displaystyle G_{k}(t)= \sum_{i=1}^{k} (-1)^{i-1}(t+k-i)^{n-1}  {n \choose i-1}.$ It follows that
\begin{equation}\label{ikpgk}
\sum_{i=p+1}^{k-1}  {k+n-i-1 \choose n-1}  (-1)^{i-p} {n \choose i-p-1}= (-1)^{k-p-1}  {n \choose k-p-1}.
\end{equation}
If $p\neq 0$, let $q=p-1$. By identity \eqref{ikpgk}, the coefficient of $(t+p)^{n-1}$ in equation \eqref{tpngk1} satisfies
\begin{equation}\label{gk1lp00}
\begin{split}
& L_{p}\left(G_{k+1}(t)\right) \\
=&\sum_{i=p+1}^{k}  {k+1+n-i-1 \choose n-1}  (-1)^{i-p} {n \choose i-p-1}\\
=& \sum_{i=q+2}^{k}  {k+1+n-i-1 \choose n-1}  (-1)^{i-q-1} {n \choose i-q-2}\\
=& \sum_{i=q+1}^{k-1}  {k+n-i-1 \choose n-1}  (-1)^{i-q} {n \choose i-q-1}\\
=& (-1)^{k-q-1}  {n \choose k-q-1}= (-1)^{k-p}  {n \choose k-p}.
\end{split}
\end{equation}
Identity \eqref{gk1lp00} holds for all integer $p \in [1, k-1]$. The only case left is $p=0$. If $p=0$, by equation \eqref{tpngk1}, the coefficient of $t^{n-1}$ in $G_{k+1}(t)$ is
\begin{equation}\label{k10gk1}
L_0 \left(G_{k+1}(t) \right)= \sum_{i=1}^{k}  {k+1+n-i-1 \choose n-1}  (-1)^{i} {n \choose i-1}.
\end{equation}
Next, we show that $\displaystyle L_0 \left(G_{k+1}(t) \right)=  (-1)^{k} {n \choose k}$. Note that by  the binomial theorem, the coefficient of $x^{k+1}$ in $(1+x)^{-n}(1+x)^{n}$ is
\begin{equation}\label{l0gk1}
\begin{split}
& \sum_{i=0}^{k}(-1)^{i} {n+i-1 \choose i}  {n \choose k-i} \\
=&\sum_{i=0}^{k}(-1)^{i} {n+i-1 \choose n-1}  {n \choose k-i}\\
=& \sum_{j=1}^{k+1}  {k+1+n-j-1 \choose n-1}  (-1)^{k+1-j} {n \choose j-1} \quad (j=k+1-i)\\
=&(-1)^{k+1} \left[ L_0 \left(G_{k+1}(t) \right)+  (-1)^{k+1} {n \choose k}  \right].
\end{split}
\end{equation}
On the other hand,  for nonnegative integer $k$, the coefficient of $x^{k+1}$ in $(1+x)^{-n}(1+x)^{n}=1$ is always $0$. Hence, \eqref{l0gk1} implies that
\begin{equation}\label{gk1l00}
L_0 \left(G_{k+1}(t) \right)= (-1)^{k} {n \choose k}.
\end{equation}
Therefore, by identities \eqref{gk1lp00} and \eqref{gk1l00}, it follows that
\begin{equation}
\begin{split}
& G_{k+1}(t) \\
=& (t+k)^{n-1} +\sum_{p=0}^{k-1} (-1)^{k-p}  {n \choose k-p} (t+p)^{n-1}\\
= & \sum_{i=1}^{k+1} (-1)^{i-1}(t+k+1-i)^{n-1}  {n \choose i-1}.
\end{split}
\end{equation}
The proof is complete.
\end{proof}

%

\section{Application to Loggamma function}\label{application}
In this section, we consider the integral
\begin{equation}\label{ndcase}
I= \int_0^1 \int_0^1 \dots \int_0^1 f(x_1+x_2 + \dots +x_n) dx_1 \, dx_2 \, \dots \, dx_n
\end{equation}
with a given function $f(x)= \log \Gamma(x)$. Integral of loggamma function has its own importance in many parts of mathematics \cite{AM, CS}. Actually, the case when $n=2$ is a problem proposed by Ovidiu Furdur \cite{OF} in ``Problems and Solutions" in The college Mathematics Journal and one of its solutions is proposed by Geng-zhe  Chang \cite{CH1}. When it comes to general dimension $n$, it is quite a challenge to simplify it. 

After the preparation of Theorem \ref{maintheorem1} in Section \ref{general}, we can evaluate the integral \eqref{ndcase} with $f(x)= \log \Gamma(x)$.  A nice formula is presented in Theorem \ref{maintheorem2}. 
\begin{theorem}\label{maintheorem2}
\begin{eqnarray}\label{loggammaformula}
I=I(n)&=& \int_0^1 \int_0^1 \dots \int_0^1 \log \Gamma(x_1+x_2 + \dots +x_n) dx_1 \, dx_2 \, \dots \, dx_n \nonumber \\
&=&\frac{1}{2} \log (2 \pi) - \frac{n-1}{2}H_n + \sum_{k=2}^{n-1} \frac{  (-1)^{n+k+1} k^n }{n!} {n-1 \choose k} \log k, \nonumber
\end{eqnarray}
where $\displaystyle H_n=\sum_{k=1}^{n} \frac{1}{k} =1+\frac{1}{2}+ \dots + \frac{1}{n}$.
\end{theorem}
The proof of this theorem is based on  Theorem \ref{maintheorem1} and several combinatorial identities in Jihuai Shi's book \cite{SH}.
Let $f(x) = \log \Gamma(x)$. By Theorem \ref{maintheorem1}, the  integral $I$ can be simplified by using several combinatorial identities. Note that $\Gamma(t+1)= t \Gamma(t)$ and $\displaystyle G_m(t)= \sum_{i=1}^{m} (-1)^{i-1}(t+m-i)^{n-1}  {n \choose i-1}$. The integral $I$ becomes
\begin{eqnarray}\label{Iformula1}
&I =& \frac{1}{(n-1)!} \sum_{m=1}^{n} \int_0^1  G_m(t) \log \Gamma(t+m-1)  dt \\
&=&  \frac{1}{(n-1)!}  \int_0^1 \sum_{m=1}^n G_m(t) \log \Gamma(t) dt +  \frac{1}{(n-1)!}  \int_0^1 \sum_{k=2}^n  \sum_{m=k}^{n} G_m(t) \log(t+ k-2) dt. \nonumber 
\end{eqnarray}

\begin{lemma}\label{id1}
\begin{equation*}
\sum_{m=k}^n G_m(t)= (n-1)!-  \sum_{m=1}^{k-1} {n-1 \choose k-m-1} (-1)^{k-m-1} (t+m-1)^{n-1},
\end{equation*}
Specially when $k=1$, $\displaystyle \sum_{m=1}^n G_m(t)= (n-1)!$. 
\end{lemma}
\begin{proof}
Note that $\displaystyle G_m(t)= \sum_{i=1}^{m} (-1)^{i-1}(t+m-i)^{n-1}  {n \choose i-1}$. It follows that
\begin{eqnarray*}
\sum_{m=1}^k G_m(t)&=& \sum_{m=1}^k  \sum_{i=1}^{m} (-1)^{i-1}(t+m-i)^{n-1}  {n \choose i-1}\\
&=& \sum_{m=1}^k  \sum_{i=0}^{k-m} (-1)^i {n \choose i} (t+m-1)^{n-1}. 
\end{eqnarray*}
By a combinatorial identity $\displaystyle \sum_{i=0}^{m} (-1)^i {n \choose i} =(-1)^m {n-1 \choose m}  \, (m <n)$, we have
\[ \sum_{m=1}^k  \sum_{i=0}^{k-m} (-1)^i {n \choose i} (t+m-1)^{n-1}= \sum_{m=1}^k {n-1 \choose k-m} (-1)^{k-m} (t+m-1)^{n-1}. \]
Hence, 
\[ \sum_{m=1}^k G_m(t) =  \sum_{m=1}^k {n-1 \choose k-m} (-1)^{k-m} (t+m-1)^{n-1}.  \]
In special when $k=n$,  combinatorial identity $\displaystyle  \sum_{k=0}^n (-1)^k {n \choose k} (x+n-k)^n = n!$ implies
\begin{eqnarray*}
\sum_{m=1}^n G_m(t) &= & \sum_{m=1}^n {n-1 \choose n-m} (-1)^{n-m} (t+m-1)^{n-1} \\
&=& \sum_{k=0}^{n-1}  {n-1 \choose k} (-1)^{k} (t+n-1-k)^{n-1}\\
 &=& (n-1)!.
\end{eqnarray*}
Therefore,
\begin{eqnarray*}
\sum_{m=k}^n G_m(t)&=&  \sum_{m=1}^n G_m(t)  - \sum_{m=1}^{k-1} G_m(t)\\
&=&  (n-1)!- \sum_{m=1}^{k-1} {n-1 \choose k-m-1} (-1)^{k-m-1} (t+m-1)^{n-1}.
\end{eqnarray*}
The proof is complete.
\end{proof}

Let 
\[ T_k= \sum_{m=1}^{k} {n-1 \choose k-m} (-1)^{k-m} (t+m-1)^{n-1}= \sum_{m=0}^{k-1} {n-1 \choose m} (-1)^{m} (t+k-m-1)^{n-1}.\] Then $\displaystyle \sum_{m=k}^n G_m(t)= (n-1)! - T_{k-1}$. By applying Lemma \ref{id1}, \eqref{Iformula1} becomes
\begin{equation}\label{formulaofI}
\begin{split}
I=&    \int_0^1 \log \Gamma(t) dt + \int_0^1 \sum_{k=0}^{n-2} \log (t+k) dt - \frac{1}{(n-1)!} \int_0^1  \sum_{k=1}^{n-1} T_k  \log (t+k-1) dt   \\
 = & \frac{1}{2} \log(2 \pi) + (n-1) \log(n-1) - n+1 - \frac{1}{(n-1)!} \int_0^1  \sum_{k=1}^{n-1} T_k  \log (t+k-1) dt. 
 \end{split}
\end{equation}
Then the calculation of $I$ is reduced to compute $\displaystyle \int_0^1  \sum_{k=1}^{n-1} T_k  \log (t+k-1) dt$. Note that $T_1=t^{n-1}$, 
\begin{eqnarray*}
&& \int_0^1 T_k  \log (t+k-1) dt\\
&=&  \sum_{m=0}^{k-1} {n-1 \choose m}  (-1)^{m} \int_0^1 (t+k-m-1)^{n-1} \log (t+ k-1) dt,
\end{eqnarray*}
and for $k >1$,
\begin{eqnarray*} 
& & \int_0^1 (t+k-m-1)^{n-1} \log (t+ k-1) dt \\
&=& \frac{(k-m)^n \log k - (k-m-1)^n \log (k-1) }{n} - \int_0^1 \frac{(t+k-m-1)^n}{n(t+k-1)} dt\\
&=& \frac{(k-m)^n- (-m)^n}{n} \log k - \frac{(k-m-1)^n- (-m)^n}{n} \log(k-1)+\\
&& \quad - \frac{1}{n} \sum_{r=1}^n \frac{k^r- (k-1)^r}{r}{n \choose r} (-m)^{n-r}.
\end{eqnarray*}

Let $S_1(1)=0$, 

\[ S_1(k)=\sum_{m=0}^{k-1} {n-1 \choose m}  (-1)^{m} \left[  \frac{(k-m)^n- (-m)^n}{n} \log k  - \frac{(k-m-1)^n- (-m)^n }{n} \log(k-1) \right],\]
 and $\displaystyle S_2(k) = \frac{1}{n} \sum_{m=0}^{k-1} {n-1 \choose m} (-1)^m  \sum_{r=1}^n \frac{k^r- (k-1)^r}{r}{n \choose r} (-m)^{n-r}$. It follows that 
 \begin{equation}\label{formulaintTk}
 \int_0^1  \sum_{k=1}^{n-1} T_k  \log (t+k-1) dt =  \sum_{k=1}^{n-1} S_1(k) - \sum_{k=1}^{n-1} S_2(k). 
\end{equation}

\begin{lemma}\label{id2}
\begin{eqnarray*}
 \sum_{k=1}^{n-1} S_1(k) = \frac{1}{n} \sum_{k=2}^{n-2} {n-1 \choose k} (-1)^k (-k)^n \log k + \frac{\log (n-1)}{n} \left[ n! (n-1) -(n-1)^n  \right].
\end{eqnarray*}
\end{lemma}

\begin{proof}
Note that $S_1(1)=0$. 
\begin{eqnarray*}
 && \sum_{k=1}^{n-1} S_1(k) \\
&=&  \sum_{k=1}^{n-1} \sum_{m=0}^{k-1} {n-1 \choose m} (-1)^m \left[  \frac{(k-m)^n- (-m)^n}{n} \log k  - \frac{(k-m-1)^n- (-m)^n }{n} \log(k-1) \right]\\
&=& \frac{1}{n} \sum_{k=2}^{n-2} {n-1 \choose k} (-1)^k (-k)^n \log k + \\
&& \quad + \frac{1}{n} \sum_{m=0}^{n-2}   {n-1 \choose m } (-1)^m \left[(n-m-1)^n- (-m)^n \right] \log(n-1).
\end{eqnarray*}
By combinatorial identity $\displaystyle \sum_{k=0}^n {n \choose k} (-1)^k  (x-k)^{n+1} = (x-\frac{n}{2}) (n+1)!$, we have  
\[ \sum_{m=0}^{n-2}    {n-1 \choose m } (-1)^m (n-1-m)^n= \sum_{m=0}^{n-1} {n-1 \choose m } (-1)^m (n-1-m)^n= \frac{n-1}{2} n!, \]
\begin{eqnarray*}
 & & \sum_{m=0}^{n-2}    {n-1 \choose m } (-1)^m (-m)^n \\
 &=& \sum_{m=0}^{n-1}    {n-1 \choose m } (-1)^m (-m)^n - (-1)^{n-1}(1-n)^n= (n-1)^n- \frac{n-1}{2} n!.
\end{eqnarray*}
Hence, 
\begin{eqnarray*}
 \sum_{k=1}^{n-1} S_1(k)= \frac{1}{n} \sum_{k=2}^{n-2} {n-1 \choose k} (-1)^k (-k)^n \log k + \frac{\log (n-1)}{n} \left[ n! (n-1) -(n-1)^n  \right].
\end{eqnarray*}
The proof is complete.
\end{proof}

The following Lemma calculates the sum $\displaystyle \sum_{k=1}^{n-1} S_2(k) $. Here we only list the result. For reader's convenience, the proof of it is presented in the Appendix A.
\begin{lemma}\label{id3}
\begin{eqnarray*}
 \sum_{k=1}^{n-1} S_2(k) = (n-1)!(n-1) - \frac{n-1}{2}H_n(n-1)!,
\end{eqnarray*}
where $\displaystyle H_n= \sum_{i=1}^n \frac{1}{i}$.
\end{lemma}

After Lemma \ref{id2} and Lemma \ref{id3}, we can show Theorem \ref{maintheorem2} in this section. \\

\textbf{Proof of Theorem \ref{maintheorem2}:} 
\begin{proof}
By identity \eqref{formulaintTk}, Lemma \ref{id2} and Lemma \ref{id3}, they imply that
\begin{eqnarray}\label{Tkint}
& &  \int_0^1  \sum_{k=1}^{n-1} T_k  \log (t+k-1) dt  \\
 &= & \sum_{k=1}^{n-1} S_1(k) - \sum_{k=1}^{n-1} S_2(k) \nonumber\\
 &=&  \frac{1}{n} \sum_{k=2}^{n-2} {n-1 \choose k} (-1)^k (-k)^n \log k + \frac{\log (n-1)}{n} \left[ n! (n-1) -(n-1)^n  \right] +\nonumber\\
 & & \qquad -(n-1)!(n-1)+ \frac{n-1}{2}H_n(n-1)!. \nonumber
\end{eqnarray}

By identities \eqref{formulaofI} and \eqref{Tkint}, it follows that
\begin{eqnarray*}
I&=& \frac{1}{2} \log(2 \pi) + (n-1) \log(n-1) - n+1 - \frac{1}{(n-1)!} \int_0^1  \sum_{k=1}^{n-1} T_k  \log (t+k-1) dt \\
& =& \frac{1}{2} \log(2 \pi) -\frac{n-1}{2}H_n + \frac{1}{n!}\sum_{k=2}^{n-1} {n-1 \choose k} (-1)^{k+n+1} k^n \log k, 
\end{eqnarray*}
where $\displaystyle H_n= \sum_{i=1}^n \frac{1}{i}$. The proof is complete.
\end{proof}

Specially, one can easily calculate the integral when $n=2,3$ and $4$:
\[I(2)= -\frac{3}{4} + \frac{1}{2} \log (2 \pi),  \]
\[I(3)= \frac{1}{2} \log (2 \pi) + \frac{4}{3} \log 2 - \frac{11}{6}, \]
\[I(4)= \frac{1}{2} \log (2 \pi)- 2 \log 2 + \frac{27}{8} \log 3 -\frac{25}{8}. \]

\begin{appendix}
\renewcommand\thesection{\appendixname~\Alph{section}}
\section{}
For reader's convenience, the proof of Lemma  \ref{id3} is presented in the appendix.

\begin{lemma}[Lemma \ref{id3}]
\begin{eqnarray*}
 \sum_{k=1}^{n-1} S_2(k) = (n-1)!(n-1) - \frac{n-1}{2}H_n(n-1)!,
\end{eqnarray*}
where $\displaystyle H_n= \sum_{i=1}^n \frac{1}{i}$.
\end{lemma}

\begin{proof}
Note that 
\begin{eqnarray*}
 &&\sum_{k=1}^{n-1} S_2(k) \\
&=& \frac{1}{n} \sum_{k=1}^{n-1} \left[  \sum_{m=0}^{k-1} {n-1 \choose m} (-1)^m  \sum_{r=1}^n \frac{k^r- (k-1)^r}{r}{n \choose r} (-m)^{n-r} \right] \\
&=& -\frac{1}{n} \sum_{k=1}^{n-2} {n-1 \choose k} (-1)^k (-k)^n \sum_{r=1}^n {n \choose r} \frac{(-1)^r}{r}+ \\
 & & \quad +\frac{1}{n} \sum_{m=0}^{n-2} {n-1 \choose m} (-1)^m \sum_{r=1}^n \frac{(-m)^{n-r} (n-1)^r }{r} {n \choose r} .
\end{eqnarray*}
Let 
\begin{equation}\label{R1}
 \displaystyle R_1= -\frac{1}{n} \sum_{k=1}^{n-2} {n-1 \choose k} (-1)^k (-k)^n \sum_{r=1}^n {n \choose r} \frac{(-1)^r}{r}
 \end{equation}
 and 
\begin{equation}\label{R2}
\displaystyle R_2= \frac{1}{n} \sum_{m=0}^{n-2} {n-1 \choose m} (-1)^m \sum_{r=1}^n \frac{(-m)^{n-r} (n-1)^r }{r} {n \choose r} , 
\end{equation}
Then 
\begin{equation}\label{S2R1R2}
\displaystyle \sum_{k=1}^{n-1} S_2(k) = R_1+R_2.
\end{equation} 

By applying combinatorial identities $\displaystyle \sum_{k=0}^n {n \choose k} (-1)^k (x-k)^{n+1} = (x- \frac{n}{2}) (n+1)!$ and $\displaystyle  \sum_{k=1}^n \frac{(-1)^{k+1}}{k} {n \choose k}= H_n$, $R_1$ can be simplified to
\begin{eqnarray}\label{R1formula} 
R_1&=& \frac{1}{n} \sum_{k=1}^{n-2} {n-1 \choose k} (-1)^k (-k)^n \sum_{r=1}^n {n \choose r} \frac{(-1)^{r+1}}{r} \nonumber \\
&=& \frac{H_n}{n} \left[ (n-1)^n - \frac{n-1}{2} n!  \right]. 
\end{eqnarray}

To simplify $R_2$, we apply combinatorial identity $\displaystyle \sum_{k=1}^n \frac{(-1)^{k+1}}{k} {n \choose k} \left[ 1-(1-x)^k\right]= \sum_{k=1}^n \frac{x^k}{k}$ and it follows that
\begin{eqnarray*}
& &\sum_{r=1}^n \frac{(-m)^{n-r} (n-1)^r }{r} {n \choose r} \\
&=&- (-m)^n \sum_{r=1}^n  \frac{(-1)^{r+1} }{r} {n \choose r} (1- \frac{m-n+1}{m})^r\\
&=& (-m)^n \left[ \sum_{r=1}^n \frac{1}{r} (\frac{m-n+1}{m})^r- \sum_{r=1}^n \frac{(-1)^{r+1}}{r} {n \choose r}  \right]\\
&=&\sum_{r=1}^n \frac{1}{r} (m-n+1)^r m^{n-r} (-1)^n -  (-m)^n H_n .
\end{eqnarray*}
Recall the formula of $R_2$ in \eqref{R2}, it follows that
\begin{eqnarray}\label{RR2}
& & n R_2\\
 &=& \sum_{m=0}^{n-2} {n-1 \choose m} (-1)^{m+n} \sum_{k=1}^n \frac{1}{k} (m-n+1)^k m^{n-k} - H_n \sum_{m=0}^{n-2} {n-1 \choose m} (-1)^m (-m)^n. \nonumber
\end{eqnarray}
Note that, by combinatorial identity $\displaystyle \sum_{k=0}^n {n \choose k} (-1)^k (x-k)^{n+1} = (x- \frac{n}{2}) (n+1)!$, we have
\begin{equation}\label{nR11}
 \sum_{m=0}^{n-2} {n-1 \choose m} (-1)^m (-m)^n= (n-1)^n- \frac{n-1}{2} n!.
 \end{equation}
To simplify $\displaystyle \sum_{m=0}^{n-2} {n-1 \choose m} (-1)^{m+n} \sum_{k=1}^n \frac{1}{k} (m-n+1)^k m^{n-k}$, we note that
\begin{eqnarray}\label{nR22}
&&\sum_{m=0}^{n-2} {n-1 \choose m} (-1)^{m+n} \sum_{k=1}^n \frac{1}{k} (m-n+1)^k m^{n-k} \\
&=& \sum_{k=1}^n \frac{(-1)^n}{k} \left[ \sum_{m=0}^{n-2} {n-1 \choose m} (-1)^{m} \sum_{i=0}^k {k \choose i} m^{n-k+i}(n-1)^{k-i}(-1)^{k-i}    \right].\nonumber
\end{eqnarray}
Let $\displaystyle P(m)= \sum_{i=0}^k {k \choose i} m^{n-k+i}(n-1)^{k-i}(-1)^{k-i} $. We apply the combinatorial identity $\displaystyle \sum_{k=0}^n (-1)^k  {n \choose k}  \mathbf{P}(k)=0 $, for any polynomial $\mathbf{P}(k)$ with $\deg P(k) <n$, and it follows that
\begin{eqnarray}\label{nR23}
&&\sum_{m=0}^{n-2} {n-1 \choose m} (-1)^{m} \sum_{i=0}^k {k \choose i} m^{n-k+i}(n-1)^{k-i}(-1)^{k-i}  \nonumber \\
&=& \sum_{m=0}^{n-2} {n-1 \choose m} (-1)^{m} P(m) \nonumber\\
&=& \sum_{m=0}^{n-1} {n-1 \choose m} (-1)^{m} P(m)- (-1)^{n-1} P(n-1) \nonumber \\
&=& \sum_{m=0}^{n-1} {n-1 \choose m} (-1)^{m}  \left[ -k (n-1) m^{n-1}+ m^n \right].
\end{eqnarray}
By combinatorial identity $\displaystyle \sum_{k=0}^n (-1)^k  {n \choose k} (x+n-k)^n= n!$, we have
\[   -k(n-1)\sum_{m=0}^{n-1} {n-1 \choose m} (-1)^{m} m^{n-1} = k(n-1) (-1)^n  (n-1)!. \]
By combinatorial identity $\displaystyle  \sum_{k=0}^n {n \choose k} (-1)^k (x-k)^{n+1}= (x-\frac{n}{2}) (n+1)!$, we have 
\[ \sum_{m=0}^{n-1} {n-1 \choose m} (-1)^{m}m^n= (-1)^{n-1} \frac{n-1}{2}  n!.  \]
Then identity \eqref{nR22} becomes
\begin{eqnarray}\label{nR24}
&&\sum_{m=0}^{n-2} {n-1 \choose m} (-1)^{m+n} \sum_{k=1}^n \frac{1}{k} (m-n+1)^k m^{n-k}  \nonumber\\
&=& \sum_{k=1}^n \frac{(-1)^n}{k} \left[ k(n-1) (-1)^n  (n-1)!+ (-1)^{n-1} \frac{n-1}{2}  n!\right] \nonumber\\
&=& n!(n-1)- \frac{n-1}{2} n! H_n,
\end{eqnarray}
where $\displaystyle H_n= \sum_{i=1}^n \frac{1}{i}$.

Hence, by identities \eqref{nR24} and \eqref{nR11},  $n R_2$ in \eqref{RR2} can be simplified to 
\begin{equation}\label{nR2f}
n R_2 = n! (n-1) - (n-1)^n H_n.
\end{equation}
That is, 
\begin{equation}\label{R2formula}
R_2= (n-1)! (n-1) -\frac{H_n}{n} (n-1)^n. 
\end{equation}
Therefore, by identities \eqref{S2R1R2}, \eqref{R1formula} and \eqref{R2formula}, it follows that  
\begin{eqnarray*}
\sum_{k=1}^{n-1} S_2(k) & =& R_1+R_2 \\
&=&  \frac{H_n}{n} \left[ (n-1)^n - \frac{n-1}{2} n!  \right]+  (n-1)! (n-1) -\frac{H_n}{n} (n-1)^n\\
&=&  (n-1)!(n-1) - \frac{n-1}{2}H_n(n-1)!,
\end{eqnarray*}
where $\displaystyle H_n= \sum_{i=1}^n \frac{1}{i}$.
\end{proof}

\end{appendix}

\section*{Acknowledgements}
The authors want to express their gratitude to the department of mathematics at Brigham Young University for its help and support. We sincerely thank Professor Tiancheng Ouyang for his invitations.

\end{document}